\RequirePackage{ifpdf}
\ifpdf
\documentclass[12pt]{amsart}
\else
\documentclass[dvipdfmx,12pt]{amsart}
\fi

\PassOptionsToPackage{linktocpage}{hyperref}

\usepackage[foot]{amsaddr}
\usepackage[nocompress,noadjust]{cite}

\usepackage{tikz}
\usepackage{tikz-cd}
\usetikzlibrary{cd}
\usetikzlibrary{positioning, arrows.meta, matrix}

\usepackage{fancybox}

\usepackage{listings}
\usepackage{docmute}
\usepackage{amssymb, enumerate, calligra}
\usepackage{amsthm}
\usepackage{amsmath}
\usepackage{amsxtra}
\usepackage{latexsym}
\usepackage{mathrsfs}
\usepackage[all,cmtip]{xy}
\usepackage[all]{xy}
\usepackage{enumitem}
\usepackage{xcolor}
\usepackage{graphicx}
\usepackage{comment}
\usepackage{mathabx,epsfig}
\usepackage{stmaryrd}

\usepackage[hypertexnames=false]{hyperref}
\usepackage[nameinlink]{cleveref}
\usepackage{autonum}

\usepackage{mathtools}

\makeatletter
\renewcommand\th@plain{\slshape}
\makeatother
\newtheoremstyle{plain}
 {2mm}
 {2mm}
 {\slshape}
 {}
 {\bfseries}
 {.}
 {.5em}
 {}
\theoremstyle{plain}
\newtheorem{theorem}{Theorem}[section]
\newtheorem{corollary}[theorem]{Corollary}
\newtheorem{lemma}[theorem]{Lemma}
\newtheorem{proposition}[theorem]{Proposition}
\newtheorem{claim}[theorem]{Claim}
\newtheorem*{claim*}{Claim}
\newtheorem{conjecture}[theorem]{Conjecture}
\newtheorem{question}[theorem]{Question}

\newtheoremstyle{definition}
 {2mm}
 {2mm}
 {\normalfont}
 {}
 {\bfseries}
 {.}
 {.5em}
 {}
\theoremstyle{definition}

\newtheorem{definition}[theorem]{Definition}

\newtheorem{remark}[theorem]{Remark}

\newtheorem*{acknowledgements}{Acknowledgements}

\crefname{section}{Section}{Sections}

\crefname{theorem}{Theorem}{Theorems}
\crefname{corollary}{Corollary}{Corollaries}
\crefname{lemma}{Lemma}{Lemmas}
\crefname{lemma}{Lemma}{Lemmas}
\crefname{proposition}{Proposition}{Propositions}
\crefname{claim}{Claim}{Claims}
\crefname{definition}{Definition}{Definitions}
\crefname{notation}{Notation}{Notations}
\crefname{problem}{Problem}{Problems}
\crefname{note}{Note}{Notes}
\crefname{remark}{Remark}{Remarks}
\crefname{example}{Example}{Examples}
\crefname{conjecture}{Conjecture}{Conjectures}
\crefname{question}{Question}{Questions}
\crefname{mainthm}{Theorem}{Theorems}
\crefname{mainprop}{Proposition}{Propositions}

\crefname{enumi}{}{}
\crefname{enumii}{}{}
\crefname{enumiii}{}{}

\crefformat{equation}{{\upshape(#2#1#3)}}

\numberwithin{equation}{section}

\def\C{{\mathbb C}}
\def\Q{{\mathbb Q}}
\def\R{{\mathbb R}}
\def\Z{{\mathbb Z}}
\def\P{{\mathbb P}}
\def\A{{\mathbb A}}

\def\QQ{\overline{\mathbb Q}}

\def\O{{ \mathcal{O}}}

\DeclareMathOperator{\pr}{pr}
\DeclareMathOperator{\Pic}{Pic}

\DeclareMathOperator{\Spec}{Spec}

\DeclareMathOperator{\codim}{codim}

\newcommand{\e}{\varepsilon}

\renewcommand{\l}{\lambda}

\newcommand{\z}{\zeta}

\newcommand{\G}{\Gamma}

\allowdisplaybreaks

\begin{document}

	\title[Arithmetic degree and Zariski dense orbit conjecture]
	{Arithmetic degree and its application to Zariski dense orbit conjecture} 
	
	\author[Yohsuke Matsuzawa]{Yohsuke Matsuzawa}
	\author[Junyi Xie]{Junyi Xie}
	\address{Department of Mathematics, Graduate School of Science, Osaka Metropolitan University, 3-3-138, Sugimoto, Sumiyoshi, Osaka, 558-8585, Japan}
	\email{matsuzaway@omu.ac.jp}
	\address{Beijing International Center for Mathematical Research, Peking University, Beijing 100871, China}
	\email{xiejunyi@bicmr.pku.edu.cn}

	\begin{abstract} 
		We prove that for a dominant rational self-map $f$ on a quasi-projective variety defined over $\QQ$,
		there is a point whose $f$-orbit is well-defined and its arithmetic degree is arbitrarily close to the first dynamical degree of $f$.
		As an application, we prove that Zariski dense orbit conjecture holds for a birational map defined over $\QQ$ 
		whose first dynamical degree is strictly larger than its third dynamical degree.
		In particular, the conjecture holds for birational maps on threefolds whose first dynamical is degree greater than $1$.
	\end{abstract}

	

	\maketitle
	\tableofcontents

\section{Introduction}

For a dominant rational map $f \colon X \dashrightarrow X$ on a projective variety defined over $\QQ$,
Kawaguchi-Silverman conjecture predicts that height growth rate along a Zariski dense orbit is equal to
the first dynamical degree of $f$. 
More precisely, 
let $L$ be an ample divisor on $X$ and let $h_{L}$ be a Weil height function associated with $L$
(we refer to \cite{HS00,La83,BG06} for the basics of height functions).
For a point $x \in X(\QQ)$, we say the $f$-orbit is well-defined if 
\begin{align}
f^{n}(x) \notin I_{f}, \quad \forall n \geq 0
\end{align}
where $I_{f}$ is the indeterminacy locus of $f$.
The set of such points is denoted by $X_{f}(\QQ)$:
\begin{align}\label{Xfdef}
X_{f}(\QQ) = \{ x \in X(\QQ) \mid f^{n}(x) \notin I_{f}, \ \forall n \geq 0\}.
\end{align}
For $x \in X_{f}(\QQ)$, the following invariant
\begin{align}
\alpha_{f}(x) := \lim_{n \to \infty} \max\{1, h_{L}(f^{n}(x))\}^{ \frac{1}{n}}
\end{align}
is called the arithmetic degree of $f$ at $x$, provided the limit exists.
By the basic properties of height functions, it is easy to see that 
the limit is independent of the choice of $L$ and $h_{L}$.
The existence of the limit is proven for surjective self-morphisms on projective varieties \cite[Theorem 3]{KS16a}
(it is stated for normal projective varieties, but the general case easily follows from normal case by taking normalization),
and for arbitrary dominant rational self-maps and points with generic orbit \cite[Theorem 1.3]{matsuzawa-note-ad-dls}.
Here an orbit is called generic if it converges to the generic point with respect to the Zariski topology.
More generally, the convergence of the arithmetic degree is proven for orbits satisfying dynamical Mordell-Lang (DML for short) conjecture.

For $i = 0,\dots, \dim X$, the $i$-th dynamical degree of $f$ is defined by
\begin{align}
\l_{i}(f) = \lim_{n \to \infty} \deg_{i,L}(f^{n})^{ \frac{1}{n}}
\end{align}  
where the $i$-th degree $\deg_{i,L}(f^{n})$ is defined as follows.
Let $\G_{f^{n}} \subset X \times X$ be the graph of $f^{n}$ and let
$p_{i} : \G_{f^{n}} \longrightarrow X$ be the projections ($i=1,2$):
\begin{equation}
\begin{tikzcd}
\G_{f^{n}} \arrow[d,"p_{1}",swap] \arrow[rd,"p_{2}"]& \\
X  \arrow[r,dashed,"f^{n}",swap] & X
\end{tikzcd}
\end{equation}
Then we define
\begin{align}
\deg_{i,L}(f^{n}) = \big( p_{2}^{*}L^{i} \cdot p_{1}^{*}L^{\dim X - i}  \big).
\end{align}
It is known that the limits exist and are independent of the choice of $L$ (see \cite{DS05,Da20,Tr20}). 
Note that by the definition and the existence of the limits, we easily see $\l_{i}(f^{m}) = \l_{i}(f)^{m}$
for all $m \in \Z_{\geq 0}$.
Another important but non-trivial fact on dynamical degrees is the log concavity.
For every $i \in \{1,\dots, \dim X -1 \}$, the following holds:
\begin{align}
\l_{i-1}(f) \l_{i+1}(f) \leq \l_{i}(f)^{2}.
\end{align}
See \cite{Tr20}.

Now let us state Kawaguchi-Silverman conjecture.
\begin{conjecture}[Kawaguchi-Silverman conjecture {\cite{Sil12,KS16b}}]\ 

\noindent
Let $f \colon X \dashrightarrow X$ be a dominant rational map on a projective variety defined over $\QQ$.
Let $x \in X_{f}(\QQ)$.
Then $ \alpha_{f}(x)$ exists (i.e.\ the limit exists), 
and if the orbit $O_{f}(x) = \{ x, f(x), f^{2}(x), \dots\}$
is Zariski dense in $X$, then $ \alpha_{f}(x) = \l_{1}(f)$.
\end{conjecture}
We refer to \cite{matsuzawa2023recent} for an introduction and recent advances on this conjecture.
It is known that for any $x \in X_{f}(\QQ)$, the limsup version of the arithmetic degree is bounded above by
the first dynamical degree \cite[Theorem 1.4]{Ma20a} \cite[Proposition 3.11]{JSXZ21}:
\begin{align}
\overline{\alpha}_{f}(x) := \limsup_{n \to \infty} \max\{1, h_{L}(f^{n}(x))\}^{ \frac{1}{n}} \leq \l_{1}(f).
\end{align}
Thus the conjecture asserts that the arithmetic degree would take its maximal value at points with dense orbit.
Although  there is no logical implications, it is natural to ask that if there is always a point $x \in X_{f}(\QQ)$
such that $ \alpha_{f}(x) = \l_{1}(f)$.
The answer is yes for surjective morphisms on projective varieties \cite[Theorem 1.6]{MSS18a} (it is stated only for smooth projective varieties,
but the proof works for any projective varieties; just find a point at which the nef canonical height does not vanish),
and also for some classes of rational maps \cite[Theorem 3]{KS14}. See \cite{MW22,SS21,SS23} for related works.
In this paper, we prove the following.

\begin{theorem}\label{mainthm:adclosetodyndeg}
Let $X$ be a projective variety over $\QQ$.
Let $f \colon X \dashrightarrow X$ be a dominant rational map defined over $\QQ$.
Then for any $\e > 0$, the set
\begin{align}\label{mainthm:adclosetodyndeg:set}
\left\{ x \in X_{f}(\QQ) \ \middle|\ \text{$ \alpha_{f}(x)$ exists and $\alpha_{f}(x) \geq \l_{1}(f) - \e$} \right\}
\end{align}
is Zariski dense in $X$.
\end{theorem}

\begin{remark}\label{rmk:densewrtadtop}
The set \cref{mainthm:adclosetodyndeg:set} is actually dense in $X(\QQ)$
with respect to the adelic topology (in the sense of \cite{Xi22}). See \cref{thm:advsdd}.
\end{remark}

\begin{remark}
We prove the same statement for quasi-projective varieties (\cref{thm:advsdd}).
The arithmetic degree and the dynamical degrees are defined as follows.
Take a projective closure $\iota \colon X \hookrightarrow X'$, i.e. open immersion into a projective variety $X'$ over $\QQ$.
Then a dominant rational map $f \colon X \dashrightarrow X$ can be regraded as that of on $X'$, denoted by $f'$.
Then $X_{f}(\QQ) \subset X'_{f'}(\QQ)$, and we define $ \alpha_{f}(x) := \alpha_{f'}(x)$ for $x \in X_{f}(\QQ)$
(see \cite[Definition 2.3]{matsuzawa2023recent}).
The well-definedness, i.e.\ independence of the embedding follows from \cite[Lemma 3.8]{JSXZ21}, the same trick as in \cref{rem:htcomparisonqproj}.
The dynamical degrees are defined in the same way: $ \l_{i}(f) := \l_{i}(f')$.
By the birational invariance of dynamical degrees (see \cite{DS05,Da20,Tr20}), 
this definition is also independent of the embedding $\iota$.
\end{remark}

In the proof of \cite[Theorem 8.4]{JSXZ21}, they find an application of the arithmetic degree to the following Zariski dense orbit conjecture.

\begin{conjecture}[Zariski dense orbit conjecture {\cite[Conjecture 7.14]{MedScan14}, cf.\ \cite[Conjecture 4.1.6]{Z06} as well}]\label{conj_zdo} 
Let $X$ be a projective variety over an algebraically closed field $k$ of characteristic zero, 
and let $f : X \dashrightarrow X$ be a dominant rational self-map.
If every $f$-invariant rational function on $X$ is constant, then there exists $x \in X_f(k)$ whose orbit $O_f(x)$ is Zariski dense in $X$.
\end{conjecture}
Here $X_{f}(k)$ is the set of points with well-defined $f$-orbit, defined in the same way as \cref{Xfdef}.
We refer to \cite{Xi22} for the history of this conjecture and known results.
We remark that the conjecture is proven when the ground field $k$ is uncountable \cite{AC08,BGR17dynRosenlicht}.
The conjecture remains open over countable fields, in particular over $\QQ$.

The idea in  \cite[Theorem 8.4]{JSXZ21} is, roughly speaking, that a point $x \in X_{f}(\QQ)$ with $ \alpha_{f}(x) = \l_{1}(f)$
must have Zariski dense orbit under some conditions on the map $f$.
Using the same idea, in \cite[Theorem C]{MW22}, the conjecture is proven for cohomologically hyperbolic birational self-maps 
on smooth projective threefolds. 
In this paper, we weaken the assumption ``cohomologically hyperbolic" to ``$\l_{1}(f) > 1$".
More generally, we prove the following.

\begin{theorem}\label{mainthm:zdoc-l3<l1}
Let $X$ be a projective variety over $\QQ$.
Let $f \colon X \dashrightarrow X$ be a birational map.
If $\l_{3}(f) < \l_{1}(f)$ (we consider this condition is vacuous when $\dim X \leq 2$), then Zariski dense orbit conjecture holds for $f$.
That is, if $f$ does not admit invariant non-constant rational functions, then
there is a point $x \in X_{f}(\QQ)$ with $O_{f}(x)$ being Zariski dense.
\end{theorem}

\begin{remark}
Under the assumption of \cref{mainthm:zdoc-l3<l1}, if $f$ does not admit invariant non-constant rational functions,
then the set of points $x \in X_{f}(\QQ)$ with Zariski dense orbit is dense in $X(\QQ)$
with respect to the adelic topology (in the sense of \cite{Xi22}). See \cref{thm:zdoc-l3<l1}.
\end{remark}

As a corollary, we have:

\begin{corollary}
Let $X$ be a projective variety of dimension three over $\QQ$.
Let $f \colon X \dashrightarrow X$ be a birational map with $\l_{1}(f) > 1$.
Then the Zariski dense orbit conjecture holds for $f$.
\end{corollary}
\begin{proof}
Since $\l_{3}(f) = 1$, the assumption of \cref{mainthm:zdoc-l3<l1} is satisfied.
\end{proof}

\noindent
{\bf Idea of the proof.}

The idea of the proof of \cref{mainthm:adclosetodyndeg} is as follows.
By a recent work of the second author \cite{xie2024algebraic}, we roughly have
\begin{align}\label{intro-recineq}
&h_{L}(f^{n+2}(x)) - (1+\e) \mu h_{L}(f^{n+1}(x)) \\
&\geq (1-\e)\l_{1}(f)(h_{L}(f^{n+1}(x)) - (1+\e) \mu h_{L}(f^{n}(x)))
\end{align}
for some $0 \leq \mu < \l_{1}(f)$ (after replacing $f$ with its iterate).
The main problem is that we do not know in general if $h_{L}(f^{n+1}(x)) - (1+\e) \mu h_{L}(f^{n}(x)) > 0$
for some $n$.
To find such a point $x$, we consider a curve $C$ such that the degrees of $f^{n}(C)$
grow as fast as possible, i.e.\ in the order of $\l_{1}(f)^{n}$.
Then for a point $x \in C(\QQ)$, we expect inequality $h_{L}(f(x)) \geq \l_{1}(f)h_{L}(x)$ hold.
This is justified for points with large height, but we also need some additional good properties of $x$,
including well-definedness of its $f$-orbit.
The latter property is satisfied for any points in some adelic open subset (in the sense of \cite{Xi22}).
We ensure the existence of $x \in C(\QQ)$ with all desired properties by proving that 
any height function associated with an ample divisor is unbounded on a non-empty
adelic open subset \cref{prop:htunbdd-on-adelicopen}.
Once we find such a point, \cref{intro-recineq} shows $ \alpha_{f}(x) \geq  (1-\e) \l_{1}(f)$.

The idea of the proof of \cref{mainthm:zdoc-l3<l1} is as follows.
By \cref{mainthm:adclosetodyndeg}, there is a point $x$ such that $ \alpha_{f}(x) > \l_{3}(f)$.
It is known that if birational $f$ does not admit invariant non-constant rational functions,
then there are only finitely many totally invariant hypersurfaces.
Thus we may assume the orbit closure $ \overline{O_{f}(x)}$ is either $X$ or
has codimension at least two.
If it is $X$, we are done.
If it has codimension $r \geq 2$, then we can show roughly 
$ \alpha_{f}(x) \leq \l_{1}(f|_{\overline{O_{f}(x)}}) \leq \l_{1 + r}(f) \leq \l_{3}(f)$,
and this is a contradiction.

\noindent
{\bf Convention.}
\begin{itemize}
\item An  \emph{algebraic scheme} over a field $k$ is a separated scheme of finite type over $k$.
\item A  \emph{variety} over $k$ is an algebraic scheme over $k$ which is irreducible and reduced.
\item For a self-morphism $f \colon X \longrightarrow X$ of an algebraic scheme over $k$ and a 
point $x$ of $X$ (scheme point or $k'$-valued point where $k'$ is a field containing $k$), the  \emph{$f$-orbit of $x$}
is denoted by $O_{f}(x)$, i.e. $O_{f}(x) = \{ f^{n}(x) \mid n=0,1,2, \dots\}$.
The same notation is used for dominant rational map $f \colon X \dashrightarrow X$ on a variety $X$ defined over $k$
and $x \in X_{f}(k) = \{ x \in X(k) \mid f^{n}(x) \notin I_{f}, \ \forall n \geq 0\}$.
Here $I_{f}$ is the indeterminacy locus of $f$.

\item
Let $f \colon X \dashrightarrow X$ be a dominant rational map on a variety $X$ over a field $k$.
For a point $x \in X_{f}(k)$, we say $(X,f,x)$ satisfies DML property if
for any closed subset $W \subset X$, the return set $\{n \geq 0 \mid f^{n}(x) \in W\}$
is a finite union of arithmetic progressions.

\item
Let $k$ be an algebraically closed field of characteristic zero.
For a dominant rational map $f \colon X \dashrightarrow X$ on a variety over $k$,
$\l_{i}(f)$ denotes the $i$-th dynamical degree of $f$  for $i = 0,\dots, \dim X$.
The cohomological Lyapunov exponent is denoted by
$\mu_{i}(f) = \l_{i}(f) / \l_{i-1}(f)$ for $i=1,\dots, \dim X$.
We set $\mu_{\dim X + 1}(f) = 0$.
\end{itemize}

\begin{acknowledgements}
The essential part of the work was done during the Simons symposium 
``Algebraic, Complex, and Arithmetic Dynamics (2024)".
The authors would like to thank Simons Foundation, Laura DeMarco, and Mattias Jonsson 
for hosting and organizing the symposium.
The authors also would like to thank Joe Silverman and Long Wang for helpful comments. 
We are grateful to the referee for the careful reading and many valuable comments.
The first author is supported by JSPS KAKENHI Grant Number JP22K13903.
The second author is supported by the NSFC Grant No.12271007.
\end{acknowledgements}

\section{Height unboundedness on adelic open sets}

In this section, we prove that a height function associated with an ample divisor 
is unbounded on a non-empty adelic open subset.
For an algebraic scheme $X$ over $\QQ$, 
the adelic topology is a topology on $X(\QQ)$ introduced by the second author in \cite{Xi22}.
The definition involves several steps, so we do not write down it here and refer to \cite[Section 3]{Xi22}
for the definition and basic properties.
The point of the topology is that it allows us to discuss analytic local properties of $\QQ$-points
(because it is defined by using $p$-adic open sets)
while keeping coarseness of Zariski topology; if $X$ is irreducible, then $X(\QQ)$ is irreducible with respect to
the adelic topology.

\begin{definition}
Let $X$ be a quasi-projective scheme over $\QQ$.
A subset $A \subset X(\QQ)$ is said to be height bounded if the following condition holds.
For any immersion $i \colon X \hookrightarrow P$ into a projective scheme $P$ defined over $\QQ$,
any ample Cartier divisor $H$ on $P$, and any logarithmic Weil height function $h_{H}$ associated with $H$,
the subset
\begin{align}
\left\{ h_{H}(i(x)) \ \middle|\ x \in A \right\} \subset \R
\end{align}
is bounded.
\end{definition}

\begin{remark}\label{rem:htcomparisonqproj}
The set is always bounded below since so is $h_{H}$.
The definition remains equivalent if we require the boundedness only for some $i \colon X \hookrightarrow P$,
$H$, and $h_{H}$.
Indeed, if $j \colon X \hookrightarrow P'$ is another immersion to a projective scheme, $H'$ is an ample Cartier divisor on $P'$,
and $h_{H'}$ is a height associated with $H'$, form the following diagram:
\begin{equation}
\begin{tikzcd}
&& P  \\
X \arrow[r] \arrow[urr, "i", bend left=20] \arrow[drr, "j", swap,bend right=20] & \overline{X} \arrow[r] \arrow[ur, "p_{1}"] \arrow[dr, "p_{2}", swap]  & P \times P' \arrow[d, "\pr_{2}"] \arrow[u,"\pr_{1}",swap]\\
&& P'
\end{tikzcd}
\end{equation}
where $ \overline{X}$ is the scheme theoretic closure of $(i,j)(X)$ in $P \times P'$.
Take $n \geq 1$ so that
\begin{align}
{p_{2}}_{*} \big(  p_{1}^{*}\O_{P}(-H) {\otimes}_{\O_{ \overline{X}}} p_{2}^{*}\O_{P'}(nH')  \big) \simeq
\big( {p_{2}}_{*}p_{1}^{*}\O_{P}(-H) \big) {\otimes}_{\O_{P'}} \O_{P'}(nH')
\end{align}
is globally generated.
Note that $p_{2}^{-1}(j(X)) = X$.
Then the base locus of $np_{2}^{*}H' - p_{1}^{*}H$ is contained in $ \overline{X} \setminus X$,
and hence $nh_{H'} - h_{H} \geq O(1)$ on $X(\QQ)$.
Similarly, there is an $m \geq 1$ such that $m h_{H} - h_{H'} \geq O(1)$ on $X(\QQ)$.
Thus we are done.
\end{remark}

We use the notation and terminologies on adelic open subsets from \cite[Section 3]{Xi22}.

\begin{proposition}\label{prop:htunbdd-on-adelicopen}
Let $X$ be a quasi-projective variety over $\QQ$ with $\dim X \geq 1$.
Let $A \subset X(\QQ)$ be a non-empty adelic open subset in the sense of \cite{Xi22}.
Then $A$ is not height bounded.
\end{proposition}

To prove this proposition, we prepare some terminologies and a lemma.

\begin{definition}
Let $K \subset \QQ$ be a number field. 
For an algebraic scheme $X$ over $K$ and $d \in \Z_{\geq 1}$, we define
\begin{align}
X(d):= \bigcup_{\substack{K \subset L \subset \QQ \\ [L:K] \leq d }} X(L) \subset X(\QQ),
\end{align}
where each $X(L)$ is regarded as a subset of $X(\QQ)$ via the inclusion $L \subset \QQ$.
\end{definition}

\begin{lemma}\label{lem:zd-adopen}
Let $X$ be a quasi-projective variety over $\QQ$ with $\dim X \geq 1$.
Let $A \subset X(\QQ)$ be a non-empty basic adelic subset in the sense of \cite[Section 3]{Xi22}.
Let $K \subset \QQ$ be a number field and $X_{K}$ a model of $X$ over $K$.
Then there is $d \in \Z_{\geq 1}$ such that $A \cap X_{K}(d)$ is Zariski dense in $X$.
\end{lemma}
This follows from the proof of \cite[Proposition 3.9]{Xi22}.
We include here a proof for the completeness.

\begin{proof}
By replacing $K$ with a finite extension and replacing $A$ with an appropriate subset,
we may assume $A$ is a basic adelic subset over $K$ with respect to $X_{K}$.
Moreover, we may assume 
\begin{align}
A = X_{K}((\tau_{i}, U_{i}),i=1,\dots, m)
\end{align}
where $\tau_{i} \colon K \longrightarrow \C_{p_{\tau_{i}}}$ are field embeddings such that
$ | \ |_{i} := | \tau_{i}(\ ) |_{\C_{p_{\tau_{i}}}}$ are distinct absolute values on $K$,
and as usual $U_{i} \subset X_{K}(\C_{p_{\tau_{i}}})$ are non-empty $p_{\tau_{i}}$-adic open subsets.
(see the beginning of the proof of \cite[Proposition 3.9]{Xi22} for more details. )
Let $K_{p_{i}}$ be the closure of $\tau_{i}(K)$ in $\C_{p_{\tau_{i}}}$.
By further replacing $K$ with a finite extension, we may assume
$U_{i} \cap X_{K}(K_{p_{i}}) \neq  \emptyset$.
Note that this in particular implies $U_{i} \cap X_{K}(K_{p_{i}})$ is Zariski dense in $(X_{K})_{K_{p_{i}}}$.

By Noether normalization, there is a non-empty open subscheme $X_{K}^{\circ} \subset X_{K}$ with 
finite \'etale morphism 
\begin{align}
\pi \colon X_{K}^{\circ} \longrightarrow V
\end{align}
to an open subscheme $V \subset \A^{d}_{K}$ of an affine space.
By taking a connected Galois \'etale covering of $V$ dominating $X_{K}^{\circ}$ (see \cite[Proposition 3.2.10]{etalecohomology-fu} for example) 
and applying it to \cite[Proposition 3.3.1]{galois-serre}, there is a thin subset $Z \subset V(K)$ such that for all $x \in V(K) \setminus Z$,
the scheme theoretic inverse image $\pi^{-1}(x)$ is integral, i.e.\ it is of the form $\Spec \text{(field)}$.

Let $W_{i} = \pi(U_{i} \cap X_{K}^{\circ}(K_{p_{i}}))$, which is a non-empty open subset of
$V(K_{p_{i}})$.
\begin{claim}\label{claim:VZW-density}
The set 
\begin{align}
(V(K) \setminus Z) \cap \bigcap_{i=1}^{m} W_{i}
\end{align}
is Zariski dense in $V$.
\end{claim}
\begin{proof}[Proof of \cref{claim:VZW-density}]
Suppose it is contained in a proper Zariski closed subset $C \subset V$.
Let $\psi \colon V(K) \longrightarrow \prod_{i=1}^{m}V(K_{p_{i}}), x \mapsto (x, \dots, x)$
be the diagonal embedding.
Then we have
\begin{align}
\psi^{-1}\bigg(  \prod_{1 \leq i \leq m} (V \setminus C)(K_{p_{i}})  \cap \prod_{1 \leq i \leq m} W_{i}\bigg) \cap (V(K) \setminus Z) =  \emptyset.
\end{align}
Since $W_{i}$ are Zariski dense in $V_{K_{p_{i}}}$, 
\begin{align}
\prod_{1 \leq i \leq m} (V \setminus C)(K_{p_{i}})  \cap \prod_{1 \leq i \leq m} W_{i}
\end{align}
is a non-empty open subset of $\prod_{i=1}^{m}V(K_{p_{i}})$.
But by the same proof of \cite[Lemma 3.11]{Xi22}, $\psi(V(K) \setminus Z)$ is dense in $\prod_{i=1}^{m}V(K_{p_{i}})$.
Thus we get a contradiction.
\end{proof}

Let $x \in (V(K) \setminus Z) \cap \bigcap_{i=1}^{m} W_{i}$.
Then $\pi^{-1}(x) = \Spec L$ for some finite field extension $L$ of $K$.
Note that $[L:K] \leq \deg \pi$.
Fixing a field embedding $L \to \QQ$ over $K$ and we get a point $z \in X_{K}^{\circ}(\QQ) \subset X_{K}(\QQ)$:
\begin{equation}
\begin{tikzcd}
X_{K}^{\circ}  \arrow[d, "\pi",swap] & \Spec L \arrow[d] \arrow[l]& \Spec \QQ \arrow[l]\\
V & \Spec K \arrow[l,"x"].
\end{tikzcd}
\end{equation}
Since $x \in W_{i} = \pi(U_{i}\cap X_{K}^{\circ}(K_{p_{i}}))$, there is  a point $y_{i} \in U_{i}\cap X_{K}^{\circ}(K_{p_{i}})$ such that
$\pi(y_{i}) = x$. Then we get the following diagram
\begin{equation}
\begin{tikzcd}
&\Spec K_{p_{i}} \arrow[d] \arrow[ld, "y_{i}",swap] & \Spec \C_{p_{\tau_{i}}} \arrow[l] \arrow[d, dashed]\\
X_{K}^{\circ}  \arrow[d, "\pi",swap] & \Spec L \arrow[d] \arrow[l]& \Spec \QQ \arrow[l]\\
V & \Spec K \arrow[l,"x"]
\end{tikzcd}
\end{equation}
where the dashed arrow is induced by extending $L \longrightarrow K_{p_{i}} \hookrightarrow \C_{p_{\tau_{i}}}$
to $\QQ \longrightarrow \C_{p_{\tau_{i}}}$. This embedding restricted on $K$ agrees with $\tau_{i}$.
This means $z \in X_{K}^{\circ}(\tau_{i}, U_{i}\cap X_{K}^{\circ}(\C_{p_{\tau_{i}}})) \subset X_{K}(\tau_{i}, U_{i})$.
Therefore we proved $z \in A \cap X_{K}(\deg \pi)$.
Since $x$ is an arbitrary element of $(V(K) \setminus Z) \cap \bigcap_{i=1}^{m} W_{i}$, which is Zariski dense in $V$,
these $z$'s are Zariski dense in $X$ and we are done.

\end{proof}

\begin{proof}[Proof of \cref{prop:htunbdd-on-adelicopen}]
We may assume $A$ is a general adelic subset, i.e.\ 
there is a flat morphism $\pi \colon Y \longrightarrow X$ from a reduced 
algebraic scheme over $\QQ$
and a basic adelic subset $B \subset Y(\QQ)$ such that $A = \pi(B)$.
By replacing $Y$ with a small open affine subscheme of an irreducible component intersecting with $B$,
we may assume $Y$ is a quasi-projective variety.
Let $K \subset \QQ$ be a number field such that $X, Y$, and $\pi$ are defined over $K$.
Let $\pi_{K} \colon Y_{K} \longrightarrow X_{K}$ be their model.
Now suppose $A$ is height bounded.
Then for all $d \in \Z_{\geq 1}$, $A \cap X_{K}(d)$ are finite sets because of
Northcott's theorem.
Since $B \cap Y_{K}(d) \subset \pi^{-1}(A \cap X_{K}(d))$, $\pi$ is flat, and $\dim X \geq 1$,
$B \cap Y_{K}(d)$ is not Zariski dense in $Y$ for all $d \in \Z_{\geq 1}$.
This contradicts \cref{lem:zd-adopen}.
\end{proof}

\begin{remark}
The proof also shows the following. Let $X$ be a quasi-projective variety over $\QQ$
and let $A \subset X(\QQ)$ be a non-empty adelic open subset.
Let $K \subset \QQ$ be a number field and $X_{K}$ a model of $X$ over $K$.
Then there is a $d \in \Z_{\geq 1}$ such that $A \cap X_{K}(d)$ is Zariski dense in $X$.

\end{remark}

\section{Arithmetic degree can be arbitrarily close to dynamical degree}

In this section, we prove \cref{mainthm:adclosetodyndeg}.
As mentioned in \cref{rmk:densewrtadtop}, 
we prove the density with respect to the adelic topology introduced in \cite{Xi22}.
Since the adelic topology is finer than the Zariski topology induced on $X(\QQ)$ (see \cite[Proposition 3.18]{Xi22}),
\cref{mainthm:adclosetodyndeg} follows.

\begin{theorem}\label{thm:advsdd}
Let $X$ be a quasi-projective variety over $\QQ$.
Let $f \colon X \dashrightarrow X$ be a dominant rational map defined over $\QQ$.
Then for any $\e > 0$, the set
\begin{align}
\left\{ x \in X_{f}(\QQ) \ \middle|\ \text{$ \alpha_{f}(x)$ exists and $\alpha_{f}(x) \geq \l_{1}(f) - \e$ } \right\}
\end{align}
is dense in  $X(\QQ)$ with respect to the adelic topology.
\end{theorem}

\begin{proof}
By replacing $X$ with its smooth locus, we may assume $X$ is smooth.
Let us take a projective closure $\iota \colon X \hookrightarrow X'$, i.e.\ $X'$ is a projective variety over $\QQ$
and $\iota$ is an open immersion. By replacing $X'$ with its normalization, we may assume $X'$ is normal.
Let $L$ be a very ample divisor on $X'$.  
We take $L$ so that the embedding $X' \hookrightarrow \P^{N}_{\QQ}$ by the complete linear system $|L|$
is not an isomorphism. 
We regard $f$ as a dominant rational self-map on $X'$.
Let us write $\l_{i} = \l_{i}(f)$ and $\mu_{i} = \mu_{i}(f)$.
Note that by the log concavity of dynamical degrees (and the convention $\mu_{\dim X + 1}=0$), we have
$\mu_{1} \geq \cdots \geq \mu_{\dim X} > \mu_{\dim X + 1} = 0$.
To prove the theorem, we may assume $\l_{1} > 1$.
Take $p \in \{1, \dots, \dim X\}$ such that
\begin{align}
\mu_{1} = \cdots = \mu_{p} > \mu_{p+1}.
\end{align}

Let $\e > 0$ be an arbitrary positive number.
Let $A \subset X(\QQ)$ be an arbitrary non-empty adelic open subset.
We will construct a point $x \in X_{f}(\QQ) \cap A$ such that $ \alpha_{f}(x) \geq \l_{1} -\e$.

Take $\z \in (0,1)$, which is close to $1$, such that
\begin{align}\label{condition-zeta}
\frac{\mu_{p+1}}{\z^{3}\mu_{p}} < 1,  \quad \z^{2}\mu_{p} > 1,\quad \z^{2}\l_{1} \geq \l_{1} - \e.
\end{align}
By \cite[Remark 3.7]{xie2024algebraic}, 
there is an $m_{\z} \geq 1$ such that for all $m \geq m_{\z}$,
\begin{align}\label{recbig}
(f^{2m})^{*}L + (\mu_{p}\mu_{p+1})^{m}L - (\z\mu_{p})^{m}(f^{m})^{*}L
\end{align}
is big as elements of $\widetilde{\Pic}(X')_{\R}$.
Here $ \widetilde{\Pic}(X')_{\R}$ is the colimit of $\Pic (X'')_{\R}$ where $X''$ runs over birational models of $X'$.
The pull-backs $(f^{2m})^{*}, (f^{m})^{*}$ are defined in a natural way using higher birational models.
See \cite{xie2024algebraic} for the detail.
We fix an $m \geq m_{\z}$ so that
\begin{align}\label{nondegcond}
\z^{2m} \mu_{p}^{m} + \z^{-2m}\mu_{p+1}^{m} \leq \z^{m}\mu_{p}^{m}
\end{align}
holds. Such $m$ exists because of \cref{condition-zeta}.
Let us fix a Weil height function $h_{L} \colon X'(\QQ) \longrightarrow \R$ associated with $L$.
We choose $h_{L}$ so that $h_{L} \geq 1$.
By \cref{recbig}, there are $c \in \R$ and a Zariski open dense subset $V \subset X$ such that
\begin{align}
&V \cap I_{f^{m}} = V \cap I_{f^{2m}} =  \emptyset\\
& h_{L} \circ f^{2m} + (\mu_{p} \mu_{p+1})^{m} h_{L} - (\z \mu_{p})^{m} h_{L} \circ f^{m} \geq c \quad \text{on $V(\QQ)$.}
\end{align}
Then by \cref{nondegcond}, we have
\begin{align}
 h_{L} \circ f^{2m} + (\z^{2m}\mu_{p}^{m}) (\z^{-2m}\mu_{p+1}^{m} )h_{L} - (\z^{2m}\mu_{p}^{m} + \z^{-2m}\mu_{p+1}^{m} ) h_{L} \circ f^{m} \geq c
\end{align}
or, equivalently
\begin{align}
 h_{L} \circ f^{2m} - \z^{-2m} \mu_{p+1}^{m} h_{L}\circ f^{m} 
 \geq 
 \z^{2m} \mu_{p}^{m} ( h_{L} \circ f^{m} - \z^{-2m} \mu_{p+1}^{m} h_{L}) + c
\end{align}
on $V(\QQ)$.
If we take $c_{1} \in \R$ so that $c_{1} - \z^{2m} \mu_{p}^{m} c_{1} = c$, then we have
\begin{align}\label{htrecineq}
 &h_{L} \circ f^{2m} - \z^{-2m} \mu_{p+1}^{m} h_{L}\circ f^{m} -c_{1}\\
 &\geq 
 \z^{2m} \mu_{p}^{m} ( h_{L} \circ f^{m} - \z^{-2m} \mu_{p+1}^{m} h_{L} - c_{1})
\end{align}
on $V(\QQ)$.
This recursive inequality almost shows that the arithmetic degree is at least $\z^{2m} \mu_{p}^{m}$.
What we need to show is that there is at least one initial point at which  $h_{L} \circ f^{m} - \z^{-2m} \mu_{p+1}^{m} h_{L} - c_{1}$
is strictly positive.
We will find such point on a curve whose forward iterates by $f^{m}$ have maximal degree growth.
But we first need to guarantee that there are plenty of points whose orbits are well-defined and 
have nice properties.

By \cite[Proposition 3.24, Proposition 3.27]{Xi22} (see proof of \cite[Proposition 3.2]{MW22} as well), 
there is a non-empty adelic open subset $A' \subset V(\QQ)$ such that
for all $x \in A'$, we have
\begin{align}
&x \in X'_{f}(\QQ),\ \# O_{f}(x) = \infty,\ O_{f}(x) \subset V,\ \text{and}\\
&\text{$(X',f,x)$ has DML property (see convention).}
\end{align}
 
Now set $g = f^{m}$. Let $d = \dim X$.
Since $\l_{1}(g) = \l_{1}(f)^{m} = \mu_{1}^{m} = \mu_{p}^{m}$, we have
\begin{align}
\lim_{n \to \infty} \big( (g^{n})^{*}L \cdot L^{d-1} \big)^{ \frac{1}{n}} = \mu_{p}^{m}.
\end{align}
(Here $(g^{n})^{*}L$ is the one defined as an element of $ \widetilde{\Pic}(X')_{\R}$. So is the intersection number.)
We choose $\eta \in (0,1)$ close to $1$ and $l \in \Z_{\geq 1}$ large enough so that
\begin{align}
&\big( (g^{l})^{*}L \cdot L^{d-1} \big) \geq (\eta \mu_{p}^{m})^{l}\\
& \eta \mu_{p}^{m} > \z^{-2m} \mu_{p+1}^{m} \\
& \frac{\eta \mu_{p}^{m}}{ (2 (L^{d}))^{ \frac{1}{l}} } > \z^{-2m} \mu_{p+1}^{m} .
\end{align}
We can choose such $\eta$ and $l$ because of \cref{condition-zeta}.

Let us pick a point $a \in A \cap A'$ such that
\begin{itemize}
\item $X'$ is smooth at $a$;
\item the projection from $a$, $p_{a} \colon \P^{N}_{\QQ} \dashrightarrow \P^{N-1}_{\QQ}$, is generically finite on $X' \setminus \{a\}$. 
\end{itemize}
(Recall we chose $L$ so that the embedding $X' \hookrightarrow \P^{N}_{\QQ}$ defined by $|L|$ is not an isomorphism.
In particular, $X'$ is not contained in a hyperplane of $\P^{N}_{\QQ}$ and thus projection from a general point of $X'$
is generically finite on $X'$. Such $a$ exists because a non-empty adelic open set is Zariski dense.)
By the choice of $A'$, we have $a \notin I_{g^{l}}$. 
Note also that $\codim I_{g^{l}} \geq 2$ since $X'$ is normal projective. 
Let $\G \subset |L|$ be the sub-linear system consisting of all hypersurfaces passing through $a$.
If $\dim X' \geq 2$, by \cref{lem:bertini}, there is an $H_{1} \in \G$ such that
$H_{1}$ is irreducible and reduced, smooth at $a$, $\dim H_{1} \cap I_{g^{l}} < \dim I_{g^{l}}$,
and $p_{a}$ is generically finite on $H_{1} \setminus \{a\}$.
If $\dim H_{1} \geq 2$, apply the same argument to the restriction of $\G$ to $H_{1}$ and get
$H_{2} \in \G$ such that $H_{1}\cap H_{2}$ is irreducible and reduced, smooth at $a$, $\dim H_{1}\cap H_{2} \cap I_{g^{l}} < \dim I_{g^{l}} \cap H_{1}$,
and $p_{a}$ is generically finite on $H_{1}\cap H_{2} \setminus \{a\}$.
Repeat this and we get $H_{1}, \dots , H_{d-1} \in |L|$ passing through $a$ such that
\begin{align}
&\text{$C := H_{1} \cap \cdots \cap H_{d-1}$ is an irreducible and reduced curve;}\\
&C \cap I_{g^{l}} =  \emptyset.
\end{align}
Moreover, the local equations of $H_{1},\dots, H_{d-1}$ form a regular sequence at each point of $C$.

Let us consider
\begin{equation}
\begin{tikzcd}
\G_{g^{l}} \arrow[d, "\pi"] \arrow[rd, " G "] & \\
X' \arrow[r, dashed, "g^{l}", swap]& X'
\end{tikzcd}
\end{equation}
where $\G_{g^{l}}$ is the graph of the rational map $g^{l}$.
Then we have
\begin{align}
\big( (g^{l})^{*}L \cdot L^{d-1} \big) = (G^{*}L \cdot \pi^{*}L^{d-1}) = (G^{*}L \cdot \pi^{-1}(C)) = \deg (g^{l}|_{C})^{*}L.
\end{align}
(Here for the second equality, we use the equality of schemes 
$\pi^{*}H_{1} \cap \cdots \cap \pi^{*}H_{d-1} = \pi^{-1}(C)$ to see that the cycle class
$c_{1}(\pi^{*}L)^{d-1} \cap [\G_{g^{l}}]$ is represented by the cycle $[\pi^{-1}(C)]$. )

Thus we get (use \cite[Theorem B.5.9]{HS00} on the normalization of $C$ to compare $h_{(g^{l}|_{C})^{*}L}$
and $h_{L|_{C} }$)
\begin{align}
h_{L} \circ g^{l}|_{C}& = h_{(g^{l}|_{C})^{*}L} + O(1) = \frac{ \deg (g^{l}|_{C})^{*}L }{ \deg L|_{C}} h_{L|_{C}} + O\big( \sqrt{(h_{L})|_{C}}\big)\\
& \geq
\frac{(\eta \mu_{p}^{m})^{l}}{(L^{d})} (h_{L})|_{C} -c' \sqrt{(h_{L})|_{C}}
\end{align}
on $C^{\circ}(\QQ)$ where $C^{\circ}$ is the normal locus of $C$ and $c'$ is a constant depends on $C, (g^{l}|_{C})^{*}L, L|_{C}, h_{L}$.
Here for the inequality, we use $\deg L|_{C} = (L^{d})$, which holds since $C$ is the complete intersection 
of members of linear system $|L|$.
 
 By the construction of $C$ (namely $a \in C$), $C^{\circ}(\QQ) \cap A \cap A'$ is 
 a non-empty adelic open set of $C(\QQ)$.
 By \cref{prop:htunbdd-on-adelicopen}, there is a point $x \in C^{\circ}(\QQ) \cap A \cap A'$ such that
 \begin{align}
 &\frac{c'}{ \sqrt{h_{L}(x)}} \leq \frac{(\eta \mu_{p}^{m})^{l}}{2(L^{d})} ;\\
 &h_{L}(g^{i}(x)) \geq M \quad i=0,\dots, l,
 \end{align}
 where $M \in \R_{\geq 1}$ is a large constant that we choose below.
 Then we have
 \begin{align}
 h_{L}(g^{l}(x)) \geq \frac{(\eta \mu_{p}^{m})^{l}}{2(L^{d})} h_{L}(x).
 \end{align}
 Thus there is an $i \in \{0, \dots, l-1  \}$ such that
 \begin{align}
 h_{L}(g^{i+1}(x)) \geq \frac{\eta \mu_{p}^{m}}{(2(L^{d}))^{1/l}} h_{L}(g^{i}(x)).
 \end{align}
 Since $x \in A'$, $g^{i}(x) = f^{mi}(x) \in V$.
 We have
 \begin{align}
&( h_{L} \circ f^{m} - \z^{-2m}\mu_{p+1}^{m}h_{L} - c_{1})(g^{i}(x))\\
&=h_{L}(g^{i+1}(x)) - \z^{-2m}\mu_{p+1}^{m}h_{L}(g^{i}(x)) - c_{1}\\
& \geq
\bigg(  \frac{\eta \mu_{p}^{m}}{(2(L^{d}))^{1/l}} - \z^{-2m}\mu_{p+1}^{m} \bigg) h_{L}(g^{i}(x)) - c_{1}\\
& \geq 
\bigg(  \frac{\eta \mu_{p}^{m}}{(2(L^{d}))^{1/l}} - \z^{-2m}\mu_{p+1}^{m} \bigg) M- c_{1}.
 \end{align}
If we chose $M$ so that this quantity is strictly positive, then we get
\begin{align}
( h_{L} \circ f^{m} - \z^{-2m}\mu_{p+1}^{m}h_{L} - c_{1})(g^{i}(x)) > 0.
\end{align}
Since $x \in A'$, we have $g^{n}(x) = f^{mn}(x) \in V(\QQ)$.
Thus by \cref{htrecineq}, we get
\begin{align}
&h_{L}(g^{i+n}(x)) - \z^{-2m} \mu_{p+1}^{m} h_{L}(g^{i+n-1}(x)) -c_{1}\\
&\geq 
(\z^{2m} \mu_{p}^{m})^{n-1} (h_{L}(g^{i+1}(x)) -\z^{-2m}\mu_{p+1}^{m}h_{L}(g^{i}(x)) - c_{1} )
\end{align}
for $n \geq 1$ and thus 
\begin{align}
\underline{\alpha}_{g}(g^{i}(x)) = \liminf_{n \to \infty} h_{L}(g^{i+n}(x))^{ \frac{1}{n}} \geq \z^{2m} \mu_{p}^{m}.
\end{align}
By \cite[Lemma 2.7]{MW22},
\begin{align} 
\underline{\alpha}_{f}(x) = \underline{\alpha}_{f}(f^{mi}(x)) = \underline{\alpha}_{g}(g^{i}(x))^{ \frac{1}{m}} \geq \z^{2}\mu_{p} = \z^{2} \l_{1} \geq \l_{1} - \e.
\end{align}
Since $(X', f, x)$ satisfies DML property, by \cite{matsuzawa-note-ad-dls}, the arithmetic degree $ \alpha_{f}(x)$ exists, i.e.
$ \alpha_{f}(x) = \underline{\alpha}_{f}(x)$. Thus we are done.
\end{proof}

\begin{remark}\label{rmk:advsd-dml}
In the setting of \cref{thm:advsdd}, the set 
\begin{align}
\label{rmkadvsdset}
\left\{ x \in X_{f}(\QQ) \ \middle|\ \parbox{15em}{$ \alpha_{f}(x)$ exists, $\alpha_{f}(x) \geq \l_{1}(f) - \e$, \\$(X, f, x)$ satisfies DML property } \right\}
\end{align}
is dense in $X(\QQ)$ with respect to the adelic topology.
This follows directly from the above proof,  or \cite[Proposition 3.27]{Xi22} and \cref{thm:advsdd}.
\end{remark}

\begin{lemma}\label{lem:bertini}
Let $X$ be a projective variety of dimension $\geq 2$ over an algebraically closed field of characteristic zero.
Let $L$ be a very ample line bundle, $a \in X$ a smooth closed point, and $W \subset X$ a closed subset such that $a \notin W$.
Suppose  $\G \subset |L|$ is a sub-linear system of the complete linear system $|L|$ 
consisting of hypersurfaces passing through $a$.
If
\begin{itemize}
\item
at least one member of $\G$ does not contain any of the irreducible component of $W$;
\item
the base locus of $\G$ is $\{a\}$;

\item
at least one member of $\G$ is smooth at $a$;
\item
the rational map $X \dashrightarrow \P^{n}$ defined by $\G$ is generically finite,
\end{itemize}
then a general member $H \in \G$ satisfies:
\begin{enumerate}
\item
$H$ is irreducible and reduced;
\item
any irreducible component of $W$ is not contained in $H$;
\item
$H$ is smooth at $a$.
\end{enumerate}
\end{lemma}
\begin{proof}
First note that since containing an irreducible component of $W$ is a closed condition,
a general member of $\G$ does not contain any irreducible component of $W$.

Next, note that the restriction of $\G$ to $X \setminus \{a\}$ has no base point.
By \cite[Corollary 3.4.9]{joinintersec}, for a general member $H \in \G$, we have $H \setminus \{a\}$ is reduced.
Moreover, since being singular at $a$ is a closed condition, for a general $H \in \G$, $H$ is smooth at $a$.
In particular, general $H \in \G$ is reduced.

Finally, since $\dim X \geq 2$, $\G$ satisfies the assumption of 
\cite[Theorem 3.4.10]{joinintersec}, and thus a general member of $\G$ is irreducible.
\end{proof}

\begin{question}
Is it possible to remove $\e$ from the statement of \cref{thm:advsdd}?
That is, are there always points $x \in X_{f}(\QQ)$ such that $ \alpha_{f}(x) = \l_{1}(f)$?
\end{question}

If there is a family of rational maps 
\begin{equation}
\begin{tikzcd}
X \arrow[rr, dashed, "f"] \arrow[rd]&[-3ex] &[-3ex] X \arrow[ld] \\
&Y&
\end{tikzcd}
\end{equation}
such that $\l_{1}(f)$ is strictly larger than any of $\l_{1}(f|_{X_{y}})$, where $X_{y}$ is the fiber over $y \in Y(\overline{\Q})$,
then the answer to the above question is no.
But we do not know if such rational map exists or not for now.
Note that over an uncountable field instead of $\QQ$, there is no such map since $\l_{1}(f|_{X_{y}}) = \l_{1}(f)$
for very general point $y$ by the product formula \cite{Da20,Tr20}
(under suitable assumptions on the family such as \cite[Definition 1.7]{xie2024algebraic}).
On the other hand, 
there is an example of a map over $\Spec \Z$ such that the dynamical degree
on the generic fiber is strictly larger than that of every special fiber \cite[Example 1.10]{xie2024algebraic}.
Although in this case, dynamical degree of the map on the total space is not defined as it is not a variety over a field.

\section{Zariski dense orbit conjecture for birational maps under certain conditions}

In this section, we prove \cref{mainthm:zdoc-l3<l1}.
We prove the following stronger statement.

\begin{theorem}\label{thm:zdoc-l3<l1}
Let $X$ be a projective variety over $\QQ$.
Let $f \colon X \dashrightarrow X$ be a birational map.
If $\l_{3}(f) < \l_{1}(f)$ (we consider this condition is vacuous when $\dim X \leq 2$), 
then Zariski dense orbit conjecture holds for $f$.
More strongly,  if $f$ does not admit invariant non-constant rational functions, then
the set
\begin{align}
\left\{ x \in X_{f}(\QQ) \ \middle|\ \text{$O_{f}(x)$ is Zariski dense in $X$} \right\}
\end{align}
is dense in $X(\QQ)$ with respect to the adelic topology.
\end{theorem}
\begin{proof}
Let us take non-empty Zariski open subsets $U,V \subset X$ such that
$U, V$ are smooth and $f$ induces an isomorphism $U \xrightarrow{\sim} V$:
\begin{equation}
\begin{tikzcd}
X \arrow[r, "f", dashed] & X \\
U \arrow[u, phantom , "\subset", sloped] \arrow[r,"\sim"] & V \arrow[u, phantom, "\subset", sloped]
\end{tikzcd}
\end{equation}
Let us consider the induced dominant rational self-map $g \colon U\cap V \dashrightarrow U \cap V$.

Suppose $f$ does not admit invariant non-constant rational function.
Then $g$ does not admit invariant non-constant rational function, either.
By \cite[Corollary 1.3]{invhs-bmt} or \cite[Theorem B]{Ca10}, there are only finitely many totally invariant hypersurfaces of $g$.
Here a hypersurface means closed subset of pure codimension one, and
a closed subset $Y \subset U\cap V$ is said to be totally invariant under $g$ if
\begin{align}
Y = \bigcup\left\{ Y' \subset \overline{g|_{U\cap V \setminus I_{g}}^{-1}(Y)} \ \  \parbox{5em}{irreducible component}\ \middle|\ \parbox{10em}{$Y'$ dominates an irreducible component of $Y$ via $g$}  \right\}
\end{align}
as sets.
Let $H \subset U\cap V$ be the union of all of such invariant hypersurfaces.

Since $g$ does not admit invariant non-constant rational function,
it is of infinite order.
Thus for any given non-empty adelic open subset $A \subset X(\QQ)$,
there is a point $x \in (U\cap V)_{g}(\QQ) \cap A \setminus H$ with infinite $g$-orbit (see \cite[Proposition 3.24]{Xi22}). 
When $\dim X \geq 3$, by \cref{thm:advsdd}, we can take such $x$ so that
$ \alpha_{g}(x) > \l_{3}(g) = \l_{3}(f)$. 

We prove $O_{g}(x)$ is Zariski dense.
Let $Z = \overline{O_{g}(x)}$ be the Zariski closure in $U \cap V$ and suppose $Z \neq U \cap V$.
Since $O_{g}(x) \cap I_{g} =  \emptyset$,  $Z \setminus I_{g}$ is dense in $Z$ and 
$g(Z \setminus I_{g}) \subset Z$.
As $g|_{U\cap V \setminus I_{g}} \colon U \cap V \setminus I_{g} \longrightarrow U\cap V$
is an open immersion, $g$ acts on the set of generic points of $Z$ transitively.
Thus $Z$ is pure dimensional and  totally invariant under $g$.
Since $x \notin H$, we have $Z \not\subset H$ and thus $\dim Z \leq d - 2$, where $d = \dim X$.
When $d \leq 2$, we are done since we took $x$ so that $O_{f}(x)$ is infinite.

Suppose $d \geq 3$. Let us fix an irreducible component $W \subset Z$ of $Z$ containing $x$ and take an $m \geq 1$ such that
$g^{m}(W \setminus I_{g^{m}}) \subset W$.
Note that since $ \alpha_{g}(x) > \l_{3}(g) \geq 1$, $\dim W > 0$.
Also $g^{m}$ is isomorphic at the generic point of $W$.
Thus by \cite[Lemma 2.3]{MW22}, we have
\begin{align}
\l_{1}(g^{m}|_{W}) \leq \l_{1 + \codim W}(g^{m}).
\end{align}
Since $\codim W \geq 2$ and $\l_{1}(g^{m}) = \l_{1}(g)^{m} > \l_{3}(g)^{m} = \l_{3}(g^{m})$, by the log concavity of dynamical degrees, we have
$ \l_{1 + \codim W}(g^{m}) \leq  \l_{3}(g^{m})$.
Thus we get
\begin{align}
\l_{1}(g^{m}|_{W})   \leq \l_{3}(g^{m}) (< \l_{1}(g^{m})).
\end{align}
Then we get
\begin{align}
\alpha_{g}(x)^{m} = \alpha_{g^{m}}(x)= \alpha_{g^{m}|_{W}}(x) \leq \l_{1}(g^{m}|_{W}) \leq \l_{3}(g^{m}) = \l_{3}(g)^{m},
\end{align}
where the first inequality follows from \cite[Proposition 3.11]{JSXZ21}.
This inequality contradicts the choice of $x$.
\end{proof}

\begin{remark}\label{rmk:setofiniptisdense}
Under the assumption of \cref{thm:zdoc-l3<l1}, when $f$ does not admit invariant non-constant rational function,
the proof actually shows the following: for any $\e>0$,
the set 
\begin{align}
\left\{ x \in X_{f}(\QQ) \ \middle|\ \text{$O_{f}(x)$ is Zariski dense in $X$ and $ \alpha_{f}(x) \geq \l_{1}(f) - \e$} \right\}
\end{align}
is dense in $X(\QQ)$ with respect to the adelic topology.
In particular, for any $\e>0$, there is an $x \in X_{f}(\QQ)$ such that $(X,f,x)$ satisfies DML property, $O_{f}(x)$ is Zariski dense in $X$, 
and $ \alpha_{f}(x) > \l_{1}(f) - \e$ (see  \cite[Proposition 3.27]{Xi22}).
In this case, the orbit $O_{f}(x)$ is generic.
By \cite[Theorem 2.2]{matsuzawa-note-ad-dls}, $ \alpha_{f}(x)$ can take only the values from $\{\l_{1}(f) = \mu_{1}(f), \mu_{2}(f), 1\}$.
Thus if we take $\e$ small enough, our point $x$ satisfies $ \alpha_{f}(x) = \l_{1}(f)$.
\end{remark}

\begin{remark}
Long Wang pointed out us that \cref{rmk:setofiniptisdense} and \cite{BDJK21} give us 
an example of a birational map with a $\QQ$-point whose arithmetic degree is a transcendental number.
Indeed, by \cite{BDJK21}, there is a birational map $f \colon \P^{3}_{\QQ} \dashrightarrow \P^{3}_{\QQ}$
whose first dynamical degree is a transcendental number.
This map $f$ does not admit non-constant rational function. Indeed, if it is the case,
the first dynamical degree of $f$ is equal to the first relative dynamical degree with respect to
a non-constant rational map to a curve, which is equal to the first dynamical degree of a very general fiber 
 (take base change to $\C$ to find such a fiber).
Since the relative dimension is two, the first dynamical degree on the fiber is algebraic, as a birational map on a surface
is always birationally conjugate to an algebraically stable map \cite[Theorem 0.1]{DillerFavre}.
This is a contradiction.
Therefore, by \cref{rmk:setofiniptisdense}, there is a point $x \in (\P^{3}_{\QQ})_{f}(\QQ)$ such that
$ \alpha_{f}(x)= \l_{1}(f)$, which is a transcendental number.

We note that the existence of transcendental arithmetic degree is first proven 
in \cite{MW22}. The map in the example was not birational, and finding an example of a birational map
was left as a problem.
Such an example was recently constructed by Sugimoto in \cite{sugimoto20241}.
The above argument gives another construction of such example.
\end{remark}

\begin{corollary}
Let $X$ be a projective variety over $\QQ$ of dimension four.
Let $f \colon X \dashrightarrow X$ be a birational map with $\l_{1}(f) \neq \l_{3}(f)$.
Then Zariski dense orbit conjecture holds for $f$.
More strongly, if $f$ does not admit invariant non-constant rational function, then
the set $\{x \in X_{f}(\QQ) \mid \text{$O_{f}(x)$ is Zariski dense in $X$}\}$ is dense in $X(\QQ)$
with respect to the adelic topology.
\end{corollary}

\begin{proof}
If $\l_{1}(f) > \l_{3}(f)$, then this is exactly the same with \cref{thm:zdoc-l3<l1}.
Suppose $\l_{1}(f) < \l_{3}(f)$. 
Since $\l_{i}(f^{-1})=\l_{4 - i}(f)$ for $i=0,\dots, 4$, we have $\l_{1}(f^{-1}) = \l_{3}(f) > \l_{1}(f) = \l_{3}(f^{-1})$.
Moreover, if $f$ does not admit invariant non-constant rational function, then neither does $f^{-1}$.
Thus by \cref{thm:zdoc-l3<l1}, the set
\begin{align}
\{x \in X_{f^{-1}}(\QQ) \mid \text{$O_{f^{-1}}(x)$ is Zariski dense in $X$}\}
\end{align}
is dense in $X(\QQ)$ with respect to the adelic topology.
As before, let us take non-empty Zariski open subsets $U,V \subset X$ such that
$f$ induces an isomorphism $U \xrightarrow{\sim} V$:
\begin{equation}
\begin{tikzcd}
X \arrow[r, "f", dashed] & X \\
U \arrow[u, phantom , "\subset", sloped] \arrow[r,"\sim"] & V \arrow[u, phantom, "\subset", sloped]
\end{tikzcd}
\end{equation}
Let us consider the induced dominant rational self-map $g \colon U\cap V \dashrightarrow U \cap V$.
Then by \cite[Proposition 3.27]{Xi22}, the set 
\begin{align}
(U\cap V)_{g}(\QQ) \cap (U\cap V)_{g^{-1}}(\QQ)
\end{align}
contains non-empty adelic open subset of $(U\cap V)(\QQ)$, hence of $X(\QQ)$.
Therefore the set
\begin{align}
\left\{ x \in (U\cap V)_{g}(\QQ) \cap (U\cap V)_{g^{-1}}(\QQ) \ \middle|\  \text{$O_{g^{-1}}(x)$ is Zariski dense in $U \cap V$} \right\}
\end{align}
is dense in $X(\QQ)$ with respect to the adelic topology.
Now we claim that any point $x$ in this set has the property $x \in X_{f}(\QQ)$ and $O_{f}(x)$ is Zariski dense in $X$.
The first property is obvious as $x \in (U\cap V)_{g}(\QQ) \subset X_{f}(\QQ)$.
By the choice of $U,V$, we see that $O_{g}(x) \subset (U \cap V)_{g^{-1}}(\QQ)$.
We note that $O_{g}(x)$ is infinite since otherwise $x$ is $g$-periodic and contradicts Zariski density of $O_{g^{-1}}(x)$. 
Let $Z \subset \overline{O_{g}(x)}$ be a top dimensional irreducible component. Here the closure is taken in $U \cap V$.
Then $(Z \cap O_{g}(x)) \setminus \{x\}$ is Zariski dense in $Z$ since $\dim Z > 0$.
Thus $g^{-1}(Z \setminus I_{g^{-1}}) \subset \overline{O_{g}(x)}$.
Thus $ \overline{g^{-1}(Z \setminus I_{g^{-1}})}$ is also a top dimensional irreducible component of $\overline{O_{g}(x)}$,
and hence we can apply the same argument.

Now take a point $y \in Z \cap O_{g}(x)$.
Then $y \in (U \cap V)_{g^{-1}}(\QQ)$ and by the above argument, we have $O_{g^{-1}}(y) \subset \overline{O_{g}(x)}$.
Since $O_{g^{-1}}(x) \subset O_{g^{-1}}(y)$ and $O_{g^{-1}}(x)$ is Zariski dense in $U \cap V$, 
we are done.
\end{proof}

\bibliographystyle{acm}
\bibliography{ad_zdoc}

\end{document}